\documentclass[pdflatex,sn-mathphys-num]{sn-jnl}

\usepackage{soul}
\usepackage{longtable}
\usepackage{booktabs}

\usepackage{hyperref}       
\usepackage{microtype}      
\usepackage{graphicx}%
\usepackage{multirow}%
\usepackage{amsmath,amssymb,amsfonts,amsthm}%
\usepackage{mathrsfs}%
\usepackage[title]{appendix}%
\usepackage{xcolor}%
\usepackage{textcomp}%
\usepackage{manyfoot}%
\usepackage{booktabs}%
\usepackage{algorithm}%
\usepackage{algorithmicx}%
\usepackage{algpseudocode}%
\usepackage{listings}%
\usepackage{caption}
\usepackage{subcaption}
\usepackage{pgfplots}
\pgfplotsset{compat=1.18}

\usepackage[noabbrev, capitalise]{cleveref}

\newtheorem*{remark}{Remark}

\title{\bf The Geometry of Manifold Learning: 
Singular foliations via Deep Learning models}

\author[1]{\fnm{Gabriele} \sur{Novelli}}\email{378046@studenti.unimore.it}

\author[2]{\fnm{Eliot} \sur{Tron}}\email{eliot.tron@univ-paris1.fr}

\author[3]{\fnm{Emanuele} \sur{Latini}}\email{emanuele.latini@unibo.it}

\author*[4]{\fnm{Rita} \sur{Fioresi}}\email{rita.fioresi@unibo.it}

\affil[1]{\orgdiv{Dipartimento FIM}, 
\orgname{Università di Modena},
\orgaddress{\street{via Campi 213/A},
\city{Modena},
\postcode{I-41125},
\country{Italy}}}

\affil[2]{
\orgdiv{SAMM},
\orgname{Université Paris 1 Panthéon Sorbonne},
\orgaddress{\street{31 rue Baudricourt},
\city{Paris},
\postcode{75013},
\country{France}}}

\affil[3]{\orgdiv{Dipartimento di Matematica}, 
\orgname{Università di Bologna},
\orgaddress{\street{Piazza Porta San Donato 5},
\city{Bologna},
\postcode{I-40126},
\country{Italy}}}

\affil[4]{\orgdiv{FaBiT}, 
\orgname{Università di Bologna},
\orgaddress{\street{via S. Donato 15},
\city{Bologna},
\postcode{I-40126},
\country{Italy}}}

\newcommand{\comment}[1]{{\color{blue}\rule[-0.5ex]{2pt}{2.5ex}}
  \marginpar{\small\begin{flushleft}\color{blue}#1\end{flushleft}}}

\usepackage{amsmath,amsfonts,bm,amsthm,amssymb}

\newtheorem{proposition}{Proposition}[section]
\newtheorem{theorem}[proposition]{Theorem}
\newtheorem{corollary}[proposition]{Corollary}
\newtheorem{lemma}[proposition]{Lemma}
\newtheorem{thm}{Theorem}[section]
\theoremstyle{definition}
\newtheorem{definition}[thm]{Definition}
\newtheorem{example}[thm]{Example}

\newtheorem{observation}{Observation}
\usepackage{booktabs} 

\newcommand{\R}{\mathbb{R}}
\newcommand{\bE}{\mathbb{E}}
\newcommand{\cF}{\mathcal{F}}
\newcommand{\cD}{\mathcal{D}}

\newcommand{\cH}{{\mathcal{H}}}
\newcommand{\fX}{{\mathfrak{X}}}

\newcommand{\ReLU}{\mathrm{ReLU}}

\newcommand{\beq}{\begin{equation}}
\newcommand{\eeq}{\end{equation}}
\newcommand{\lra}{\longrightarrow}
\newcommand{\rank}{\mathrm{rank}}
\newcommand{\im}{\mathrm{Im}}
\newcommand{\Span}{\mathrm{span}}
\newcommand{\KL}{\mathrm{KL}}

\usepackage{tikz,pgfplots}
\pgfplotsset{compat=1.18}

\begin{document}

\maketitle

{\bf Abstract.}
Deep learning algorithms have become widely used and successful across a wide range of tasks. 
A key challenge is determining whether these algorithms truly capture the intrinsic geometry of their training data in order to better understand what they learn from data.
In this theoretical work, we explore a new geometric structure defined on the data space by a trained neural network through the Data Information Matrix (DIM), a variation of the Fisher Information Matrix.
This new geometric structure is a foliation
and the main focus of this work is to explore the non-smooth case of ReLU activation functions, where traditional geometric tools often fail.
The foliation is defined through a distribution spanned by the columns of the DIM.
We first prove that this distribution is non-smooth in the sense of distributions, rendering the existence of an associated foliation much more complex.
To circumvent this issue, we introduce a smoothed version of this distribution that differs from the original only on a set of measure zero.
Finally, we prove that this smoothed distribution induces a singular foliation on the data space, which is based on the data the network has been trained with.
Our theoretical contributions provide a geometric framework to better understand how neural networks organize their input space in non-smooth settings, bridging differential geometry and deep learning theory.


\section{Introduction}
\label{introduction}
Modern machine learning performs analysis on data laying in a high dimensional data space and
yet empirical observations suggest that the meaningful data occupy only
a small portion of the data space, a concept going under
the general loose terminology of \emph{manifold hypothesis} \citep{burges, fefferman}.
Understanding the geometry of these low-dimensional structures containing
valid data in the data space has then become
a key question \citep{Sun2014AnIG}. Existing approaches face this problem studying directly the geometry of data
\citep{burges, lee, tenenbaum}, however more recently a different perspective has emerged in
\citep{tf2024, gf}: the geometrical structure of a foliation is
induced in the data space by a distribution obtained from the parameters
of a (partially) trained neural network. Indeed, the gradients of the probability distribution
defined by a trained classifier, can be used to define {a distribution, strictly
related with the dataset the classifier was trained with}. This is not
surprising, indeed {such gradients} hold
information on how the data is distributed in the data space.  This was observed
from the very beginning of the theory in \citep{LeCun1998GradientbasedLA} and it is for example
the idea on which integrated gradients
explainability results are based, starting from the seminal
\citep{intgrad} to more recent works as \citep{intgrad2}.

The purpose of this paper is to explain how the notion of
singular foliation is fundamental to understand the space on which
a dataset is embedded in and how we can exploit it towards the manifold
hypothesis in machine learning, by providing the space of data
with a natural foliation structure, generalizing the learning manifold perspective.
This structure arises from the data-information-matrix (DIM), an analogue in data space of the
Fisher information matrix (FIM). The importance of such matrices is well
established: previous works like \citep{amari3} tie the geometry of parameter space
with the learning dynamic, by showing
that the FIM induces a highly anisotropic geometry in parameter space,
which is nearly flat in most directions and strongly distorted along a few
directions, and that its largest eigenvalue controls the steepest direction
relevant to gradient-descent dynamics. 
{These ideas were furtherly developed in many papers, see \citep{nielsen-sun17b, martens} and refs. therein}. 
As far as its data space
counterpart (the DIM) is concerned, recent works show
that the structure emerging by taking the distribution generated by its columns
is tied to the amount of noise associated with a point in the data space \citep{tf2024}. However,
the full geometric properties of this foliation remain unexplored: key questions
concerning the non-smoothness of the distribution and the mechanisms
underlying its observed loss of rank are still not fully understood. Recent
developments have further demonstrated the applicability of this geometric
perspective. In particular, it has been shown in \citep{tf2024bis} that,
by applying the framework of Cartan geometry to the emerging DIM foliation structure,
it is possible to identify optimal perturbations for crafting
adversarial attacks \citep{cpt} {beyond linearity \citep{Goodfellow2014ExplainingAH}.}
Singular Riemannian geometry has also been used to study 1-D equivalent classes in \citep{benfenatiSingularRiemannianGeometry2023, benfenatiSingularRiemannianGeometry2023a, benfenatiSingularRiemannianGeometry2024}.

The present work proposes that a framework going beyond the manifold hypothesis,
exploiting the theory of singular distributions and singular foliations to provide
the data space with a natural geometric structure.
This is especially relevant for ReLU networks, whose piecewise-linear structure naturally
leads to distributions with varying rank and limited regularity.
The ReLU activation function is indeed very popular in deep learning and the complexity of such networks have been explored for instance by \citet{haninComplexityLinearRegions2019, haninDeepReLUNetworks2019}.

This paper is organized as follows.

In Sec. \ref{ig-sec}, we recall known facts regarding information
geometry and the Fisher and Data Information Matrices (FIM, DIM). More details are
found in \citep{amari, amari2, nielsen, jost, martens}.

In Sec. \ref{distr-sec}, we briefly summarize the theory of singular
distributions and we show in Thm. \ref{main1}, App. \ref{smooth-app}
that the DIM naturally defines a 
non smooth distribution for a ReLU network convolutional classifier. For more
details on prerequisites see \citep{agrachev, hermann, lavau, mckay}.

In Sec. \ref{foliat-sec}, we briefly summarize the theory of singular
foliations, we recap the main key results:
Frobenius theorem 
and Stefan-Sussmann theorem \citep{stefan, tf2024}.
We then see a key application of such results
to ReLU network convolutional classifiers and we see
the integrability of the smooth distribution $\hat \cD$ Prop. \ref{Dint}
obtained from ReLU network
convolutional classifiers and the DIM distribution.

\section{Preliminaries on Information Geometry} \label{ig-sec}

We collect definitions and results of information geometry 
\citep{amari, nielsen} to establish our notation, as well as some 
results from \citep{gf, tf2024}.
 
We denote by $p(y|x,w)$ a discrete probability distribution on 
$C$ classification labels, i.e.
{
$p_i(y|x,w)$ denotes the predicted probability that $x$'s label is $y=i$,
where $w \in \R^n$ {are the weights in} the parameter space,
$x \in \Sigma \subset \R^d$ {the input in} the data space 
and 
$y\in \{1, \dots, C\}$ {are the labels}.
$p(y|x,w)$ can thus be though as a vector of probability $(p_i(y|x,w))_{i=1,\ldots, C}$.
}



We define following \citep{amari} Sec. 3, the
\textit{information loss} $I(x,w)$ of a probability distribution as: 
$$
I(x,w):= - \log(p(y|x,w)), 
$$
We also define the \textit{loss function} $L(x,w)$ to compare
two probability distributions $p$ and $q$:
$$
L(x,w):=\bE_{y\sim q}[I(x,w)]
$$
{where $\bE$ as usual denotes the expectation value.}
In Deep Learning algorithms
$L(x,w)$ is used to compare the model output empirical distribution $p(y|x,w)$
depending on the parameters $w$ of the model,
with the known ground truth distribution $q(y|x)$, often a probability mass
distribution.
%
For such key example, $L(x,w)$ is the Kullback-Leibler (KL) divergence
between $p(y|x,w)$ and $q(y|x)$,
up to the constant $-\sum_i q_i(y|x) \log q_i(y|x)$, which is zero for
a probability mass distribution (see the work of \citet{nielsen, amari}):
$$
\begin{array}{rl}
L(x,w) & =\bE_{y \sim q}[- \log(p(y|x,w))]= \\ \\
&= \sum_{i=1}^C q_i(y|x) \log  \frac{q_i(y|x)}{p_i(y|x,w)} -
\sum_{i=1}^C q_i(y|x) \log q_i(y|x)= \\ \\
&=\KL(q(y|x)||p(y|x,w)) -
\sum_{i=1}^C q_i(y|x) \log q_i(y|x).
\end{array}
$$

In Deep Learning (DL) classification algorithms a popular choice for $p$ is
via the softmax function:
\beq\label{prob}
p_i(y|x,w)=\mathrm{softmax}(s(x,w))_i := \frac{e^{s_i(x,w)}}
{\sum_{j=1}^{C} e^{s_j(x,w)}}, 
\eeq
where $s(x,w) \in \R^C$ is the score function determined by parameters $w$.
Popular choices for $s$ involve composition of linear and non linear maps
as  ReLUs (Rectified Linear Unit).
From such $p(y|x,w)$
we derive the \textit{cross-entropy loss function}:
$$
\begin{array}{c}
L(x,w)=\bE_{y\sim q}[I(x,w)]=
- s_{y_x}(x,w) + \log{\sum_{j=1}^C e^{s_j(x,w)}}, \label{softmax}
\end{array}
$$
where $L(x,w)$ is computed with respect to
the probability mass distribution $q(y|x)$ assigns $1$ to the correct label 
$y_x$ of $x$ and zero to the other labels. 

We define the \textit{Fisher information matrix} $F$ (FIM) and the
\textit{data information matrix} $D$ (DIM) as (see the work of \citet{amari, nielsen, gf}):
\begin{align}
  F(x,w) &= \bE_{y\sim p}[\nabla_w\log p(y|x,w)
    \cdot (\nabla_w \log p(y|x,w))^T] \label{fisher} \\[1em]
D(x,w) &= \bE_{y\sim p}[\nabla_x\log p(y|x,w) \cdot
(\nabla_x \log p(y|x,w))^T]. \label{metric-g}
\end{align}

We have the following result, giving $C$, the number
of classes, as an upper bound for
the rank of both $F(x,w)$ and $D(x,w)$  (see the work of \citet{gf}). 



\begin{proposition} \label{prop1}
Let the notation be as above.
Then:
\begin{enumerate}
\item \label{thm:positive} $F(x, w)$ and $D(x, w)$ are
  positive semidefinite symmetric matrices.

\item \label{thm:ker} $\ker F(x,w)= (\Span_{i=1, \ldots, C}\{
\nabla_w \log p_i(y|x,w)\})^\perp$;\\
$\ker D(x,w)= (\Span_{i=1, \ldots, C}\{
\nabla_x \log p_i(y|x,w)\})^\perp$.

\item \label{thm:rk} $\rank\ F(x,w) < C$, \quad $\rank\ D(x,w) < C$.
\end{enumerate}
\end{proposition}

We now focus on the DIM $D$ and its
geometric significance in the data space. From
Prop. \ref{prop1} and some elementary geometric considerations we have:
\begin{equation}
\mathrm{Im}\, D(x,w)=\mathrm{span}\{ \nabla_x p_i(y|x,w)\} 
\end{equation}
Hence
\begin{equation}\label{key-obs}
\rank \,D(x,w) = \mathrm{dim} \,\mathrm{span}\{ \nabla_x p_i(y|x,w)\}
\end{equation}
This shows that the probability  distribution
$p(y|x,w)$ is not changing, if we move in the data domain
along the directions in $\ker D(x,w)$, i.e. orthogonal
to $\mathrm{Im}\, D(x,w)$, where orthogonality
is with respect to the euclidean product in $\R^d$.
We call such directions the \textit{noise directions}.
Their significance is expressed 
in Fig 1 taken from \citep{gf}, where we examine the example of a MNIST image.
\begin{figure}[h!]
\begin{center}
\includegraphics[width=\textwidth]{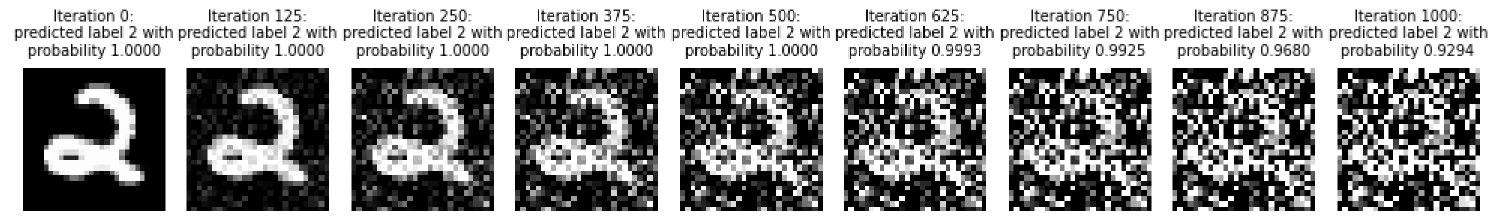}
\caption{Moving along noise directions makes images increasingly noisier,
  while the classifier continues to assign a high probably to the initial class
  \citep{gf}.}
\end{center}
\end{figure}
It is then most natural to define, starting from the data matrix $D$
the following assignments 
$\cD$ and $\cD^\perp$ in the data domain:
\beq \label{distr-def}
\begin{array}{c}
  x \mapsto \cD_x =\mathrm{Im}\, D(x,w)=
  \mathrm{span}\{\nabla_x p_i(y|x,w), \, i=1, \dots C\}
\subset \R^d  \\ \\
x \mapsto \cD_x^\perp =\ker D(x,w) = \mathrm{span}\{\nabla_x p_i(y|x,w), \, i=1, \dots C\}^\perp
\subset \R^d
\end{array}
\eeq

We define $\cD$ the \emph{DIM distribution}.
It will turn out to be important in the sequel.


\section{Singular and non-smooth Distributions} \label{distr-sec}

Let $M$ be a smooth manifold and $\fX(M)$ the
space of smooth vector fields on $M$.
We define distributions according to
\citet{stefan, sussmann}. For more details on distributions in
general see App. \ref{app-distr}.

\begin{definition}\label{distrib-def}
{\rm A \textit{distribution} 
$\Delta$ on $M$ is an assignment associating to 
$x\in M$ a vector subspace 
$\Delta _{x}\subset T_{x}M$. We say $\Delta$ is \textit{regular} or
\textit{constant rank}
if dim $\Delta_{x}=k$ for all $x \in M$. If $\Delta$ is not regular
we say it is \textit{singular}.

We say $\Delta$ is \textit{smooth}
if for each $x$ there exists a neighbourhood
$U$ of $x$ such that $\Delta_y=\Span\{X^1_y, \dots X^r_y\}$ for
all $y \in U$, $X^i\in \fX(U)$, that is, they are smooth vector fields.}
\end{definition}

A smooth distribution
can be singular, i.e. non constant rank.

\medskip
The assignment $\cD$ in (\ref{distr-def}) is a distribution.
The score function $s$ appearing
in the definition (\ref{prob}) of $p_i(y|x,w)$, for DL models  is typically a sequence
of linear maps and non linear functions applied, to each vector entry, as the ReLU:
$$
\mathrm{ReLU}(x)=\begin{cases} 0 & \hbox{if} \, x<0 \\ x & \hbox{if} \, x\geq 0 \end{cases}
$$
which is continuous but not
differentiable. Then, it is not clear whether the resulting distribution $\cD$ defined in
(\ref{distr-def}) is regular and smooth. Despite the fact $\cD$ is generated by non smooth
vector fields, the question is non trivial, as the next observation clarifies.


\begin{observation}\label{obs-sm}
A distribution holds information only \textit{at the point $x$}. So
we may have smooth vector fields spanning at every point $x$ the same 
subspace of $T_xM$ spanned by non smooth vector fields. 
For example,
consider in $\R^2$ the distribution $\Delta=\Span\{X\}$ for
$$
X=\begin{cases} 0 & \text{for $x\leq 0$} \\ \partial_x & \text{for $x > 0$}
\end{cases}
$$

Clearly this vector field is non smooth, however one can see
the same distribution $\Delta$ as equivalently generated by the smooth
vector field
$$
X'=\begin{cases} 0 & \text{for $x\leq 0$} \\ e^{-\frac{1}{x^2}}
\partial_x & \text{for $x > 0$}
\end{cases}
$$

Hence $\Delta$ is smooth. Examples of these two vector
fields can be found on Fig.~\ref{fig:vector_fields_1}.
\begin{figure}[ht]
\centering
\begin{subfigure}[b]{0.45\textwidth}
  \centering
  \begin{tikzpicture}
   \begin{axis}[%
    view     = {0}{90}, 
    axis lines=center,
    axis on top, 
    domain   = -5:5,
    y domain = -5:5,
    xtick    = {0},
    ytick    = {0},
    width    = \textwidth, 
   ]
   \addplot3[blue, scatter, only marks, domain=-5:0, samples=7, mark size=1.5pt] (x,y,0);
   \addplot3[blue, quiver={u={1}, v={0}, scale arrows=0.7}, samples=7, samples y=7, domain=0.3:5, -latex] (x,y,0);
   \end{axis}
  \end{tikzpicture}
  \caption{Vector field $X$.}
  \label{fig:field_x}
\end{subfigure}
\hfill 
\begin{subfigure}[b]{0.45\textwidth}
  \centering
  \begin{tikzpicture}
   \begin{axis}[%
    view     = {0}{90}, 
    axis lines=center,
    axis on top, 
    domain   = -5:5,
    y domain = -5:5,
    xtick    = {0},
    ytick    = {0},
    width    = \textwidth,
   ]
   \addplot3[blue, scatter, only marks, domain=-5:0, samples=7, mark size=1.5pt] (x,y,0);
   
   \addplot3[blue, 
      quiver={u={exp(-1/(x*x))}, v={0}, scale arrows=0.7}, 
      samples=8, samples y=7, domain=0.5:5, -latex] (x,y,0);
   \end{axis}
  \end{tikzpicture}
  \caption{Vector field $X'$.}
  \label{fig:field_x_prime}
\end{subfigure}

\caption{Non-smooth and smooth vector fields spanning the same distribution $\Delta$.}
\label{fig:vector_fields_1}
\end{figure}
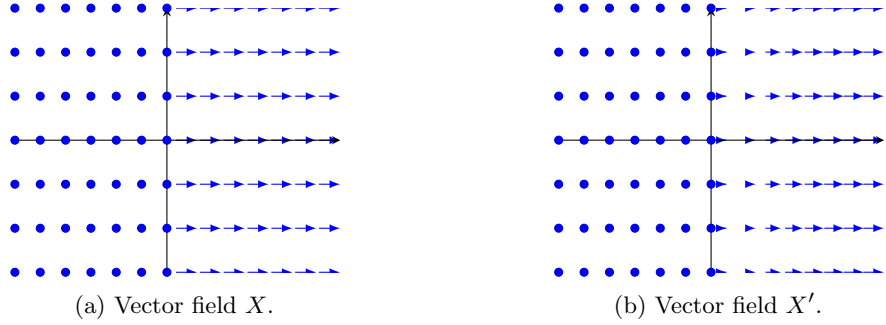

There are however distributions that are spanned by non smooth
vector fields that are truly non smooth, i.e. we cannot find locally
smooth vector fields generating it. For example,
consider the distribution in $\R^2$:
$$
\cD=\Span\{X''\}, \qquad
X''=\begin{cases} \partial_x &
\text{for $x \leq 0$} \\ \partial_x+\partial_y & \text{for $x > 0$}
\end{cases}
$$ 
This vector field is non smooth (see Fig.~\ref{fig:non_smooth_distribution_example})
and one can show that no
smooth vector field can generate the same distribution.

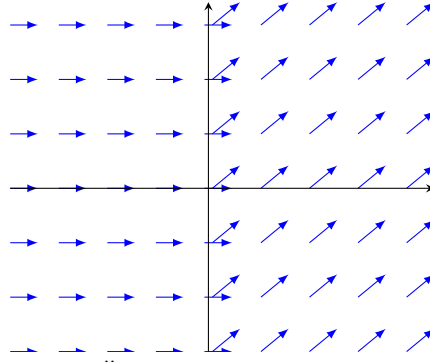
\begin{figure}[ht]
\centering
\begin{tikzpicture}
 \begin{axis}[%
  view     = {0}{90}, 
  axis lines=center,
  axis on top, 
  domain   = -5:5,
  y domain = -5:5,
  xtick    = {0},
  ytick    = {0},
  width    = 0.55\textwidth, 
]
\addplot3[blue, 
    domain=-5:-0.1,
    quiver={u={1}, v={0}, scale arrows=0.7},
    samples=5, samples y=7, -latex] (x,y,0);
\addplot3[blue, 
    domain=0.1:5,
    quiver={u={1}, v={1}, scale arrows=0.7},
    samples=5, samples y=7, -latex] (x,y,0);
\end{axis}
\end{tikzpicture}
    \caption{Vector field $X''$ spanning the non-smooth distribution $\cD$.}
    \label{fig:non_smooth_distribution_example}
\end{figure}

\end{observation}

We now focus on the DIM distribution $\cD$, as in (Eq. \ref{distr-def}).
From now on, we also assume the
score function to be given as:
\beq\label{score-def}
s(x,w) = L_{W_\ell}(w) \circ \ReLU_{\ell} \circ \dots \circ \ReLU_1 \circ L_{W_1}(w)(x), \qquad
x\in \R^d, \quad w \in \R^n
\eeq
where $w = (W_1,\ldots,W_\ell)$ are the weight matrices, $L_{W_i}(w)$ are linear maps {corresponding to the matrix multiplication 
of the output of the $i$-th layer by $W_i$} and $\ReLU_{k}$ is the ReLU function applied
entry-wise to $L_k(w) \circ \ReLU_{k-1} \dots \circ \ReLU_1 \circ L_{W_1}(w)(x)$.

{
Consider the DIM distribution in $\R^d$:  
$$
x \mapsto \cD_x = \mathrm{span} \{\nabla_x p_i(y|x, w) \, |\, i = 1, \dots, C\}
$$
We take the convention, as customary for some neural networks implementations,
to assign to the non differentiable point of each ReLU the value 1, so that, with 
an abuse of notation, we can compute $\nabla_x p_i(y|x, w)$ also in such points.
}

\medskip

The next proposition is proven in \citep{tf2024}.

\begin{proposition}\label{tf-res}
Let the notation be as above. Consider the DIM distribution in $\R^d$:  
$$
x \mapsto \cD_x = \mathrm{span} \{\nabla_x p_i(y|x, w) \, |\, i = 1, \dots, C\}
$$
and assume $p$ to be given via softmax as in Eq.~\ref{prob} and score from Eq.~\ref{score-def}.
Then, its singular points 
are a closed null subset of $\R^d$
contained in a finite union hyperplanes $\cH$.
\end{proposition}

In \citep{tf2024} the authors also show that the possible non smooth points of
the distribution are contained in the set of singular points, 
though in the light of Obs. \ref{obs-sm} the fact $\cD$ is generated by
non-smooth vector fields does not imply it is non smooth.

We now show that indeed the DIM distribution $\cD$ is non smooth.

\begin{observation}\label{score-distr}
The distribution $\mathcal{D}$ is spanned by the gradients of the
empirical probabilities $p_i(y|x,w)$ given via softmax (see Eq.~\ref{prob}):
$$
{\frac{\partial p_i(y|x,w)}{\partial x^j}}=\sum_k p_i(y|x,w)(\delta_{ik}-p_k(y|x,w))
{\frac{\partial s_k(x,w)}{\partial x^j}}
$$
where $s(x,w)=(s_1(x,w),\dots,s_C(x,w))$ is the score function,
and $0<p_i(y|x,w)<1$ because of the softmax layer.
{
We can write more compactly:
$$
\mathbb{J}(p)=A\mathbb{J}(s)\hbox{ where }A=P-pp^T\hbox{ and }P=\begin{pmatrix}
        p_1&0...&0\\
        0&p_2&...&0\\
        ...&...&...&...\\
        0&0&...&p_C
    \end{pmatrix}
$$
Consider then the matrix $B=\mathbb{I}-{1\over C}\mathbf{1}\mathbf{1}^T$ with $\mathbf{1}$ the vector with all components equals to $1$.
Then $\hbox{Ker}(B)=\mathbf{1}=\hbox{Ker}(A)$ and since $B$ is symmetric, by the Spectral
theorem we have $\hbox{Im}(B)=\hbox{Im}(A)$. Hence we have:
$$
\hbox{Im}(\mathbb{J}(p)^T)=\hbox{Im}(\mathbb{J}(s)^TA)=
\hbox{Im}(\mathbb{J}(s)^TB)
$$
But since $B$ is a constant matrix, as far as the smoothness of $\mathcal D$ is
concerned we can replace $\cD$ with $\cD_s$:
$$
\mathcal{D}_s=
\mathrm{span}\{\nabla_x s_i(x,w)\,|\, 
i=1,...,C\}
$$
}
\end{observation}

\begin{theorem}\label{main1}
Let the notation be as above. Then for a generic choice of
weights in the parameter space $\R^n$, the distribution $\cD$ is non smooth.
\end{theorem}

\begin{proof} See the App. \ref{smooth-app}. \end{proof}

{
\begin{remark}
{\rm This result states that once we fix the weights $w \in \R^n$
generically (in a closed subset of $\R^n$, see App. \ref{smooth-app}),
then the distribution $\cD$ defined on the {\sl data space}
$\R^d$ is non smooth. This result is not surprising and yet it is elusive, because
it is true not for every choice of the weights, but for a generic one.

Moreover, the example in Obs. \ref{obs-sm} shows that a distribution generated by
non smooth vector fields can be indeed smooth, thus expressing the subtlelty of
the non smoothness statement.}
\end{remark}
}

Though the DIM distribution $\cD$ is in general non smooth, we can
slightly modify it, by changing its values on a closed set in the data space, so to
make it smooth. From a machine learning perspective it is an immaterial change,
since it occurs in a measure zero set in the data space, however it is of key importance
for manifold learning, since it will bring along {\sl integrability} as we shall see
in the next section.


\begin{definition}\label{Dhat-def}
{\rm Let $\mathcal{D}$ be the DIM distribution and let $\cH$
the union of the hyperplanes containing the singular points of $\cD$. We define the
distribution $\widehat{\mathcal{D}}$ in the data space $\R^d$ as follows:
$$
\widehat{\mathcal{D}}_x=\begin{cases}
0 &  \hbox{for }
x \in\cH\\
\mathcal{D}_x=\mathrm{span}\{\nabla_x p_i(y|x,w), \, i=1, \dots C\} & \hbox{otherwise}
\end{cases}
$$}
\end{definition}


\begin{proposition}
Let the notation be as above. Then the distribution $\widehat{\cD}$ is smooth.
\end{proposition}

\begin{proof}
The aim is to find smooth vector fields that span this distribution.
Clearly for $x \not\in \cH$ (open set) there is nothing to show, $\widehat{\cD}$ is smooth,
by Prop. \ref{tf-res} and \citep{tf2024}.
Let $f_k(x)=0$ be defining the hyperplane $H_k \in \cH$ (finite union of such
hyperplanes). Let us define the smooth functions:
$$
\rho_k(x)=\begin{cases}
			\exp(-1/f_k^{{2}}(x)) \hbox{ for }f_k(x) \neq 0\\
			0 \hbox{ otherwise}
\end{cases}
$$
Let us define the following vector fields $X^{\mu}$ in the data space $\R^d$:
{$$ (X^{\mu}_x)_j=\left(\prod_{k=1}^{K}\rho_{k}(x)\right)
\mathbb{J}_x(s)^\mu_j, \qquad j=1, \dots, d $$}
where $\mathbb{J}_x(s)^\mu_j$ is the $j$-th entry of the
$\mu$-th row of the Jacobian  $\mathbb{J}_x(s)$
of the score function at $x$, {and $K$ is the total number of hyperplanes}.
Clearly 
{$X^{\mu}$} are smooth vector fields in $\R^d$.

Let us define $\widetilde{D}$ as the distribution spanned by the smooth
vector fields {$X^{\mu}$}.
Clearly $\widetilde{D}_x=\widehat{D}_x=0$ for $x \in \cH$. By Obs. \ref{score-distr}
we have that $\widetilde{D}_x=\widehat{D}_x$ for $x \not\in \cH$ and this concludes
the proof.
\end{proof}

In the next section we tackle the question of integrability of the
singular smooth distribution $\widehat{\cD}$.


\section{Integrability of the DIM distribution}\label{foliat-sec}

A natural question for a smooth distribution on a manifold $M$
regards its {\sl integrability}
i.e. whether the distribution
assignment at each point $x \in M$ of a subspace of the tangent space $T_xM$
coincides (locally) with the tangent bundle of a submanifold of $M$.
We briefly recap the main results and then we show integrability for
the singular smooth distribution $\widehat{\cD}$, introduced in the previous section.

\begin{definition}
{\rm We say that a distribution $\Delta$ is (locally) \textit{integrable}
if through all $x \in M$ there exists a submanifold $L$ of $M$ with $TL=\Delta|_L$.}
\end{definition}

Notice that in the definition we have that $\Delta$ must have constant
rank along the submanifold $L$, though in general over $M$ the rank may not be constant.
Also, we recall that the rank has a lower semicontinuity property.

\medskip
In case of constant rank, i.e. regular distributions, the main result
on integrability is given by Frobenius Theorem.

\begin{definition}
{\rm We say that a smooth distribution $\Delta$ is \textit{involutive} if for every point 
$ x\in M$ there exist smooth vector fields $\{X_{1},\ldots ,X_{n}\}$ such
that $\Delta_y=\mathrm{span} \{X_{1,y},\ldots ,X_{n,y}\}$ 
and the Lie bracket $[X_{i},X_{j}]_y \in \Delta_y$ for $y$ in a neighbourhood of $x$.}
\end{definition}

\begin{theorem}[Frobenius Theorem]
Let $\Delta$ be a smooth {regular} distribution on a manifold $M$. Then $\Delta$ is integrable
if and only if it is involutive.
\end{theorem}

For singular distributions, we have a generalization of Frobenius Theorem
that we will use in the sequel.

\begin{theorem}[Stefan-Sussmann]
A distribution $\Delta$ is locally integrable if and only if
$\Delta$ is invariant under the flow of the vector fields in $\Gamma(\Delta)$ i.e.
$\Delta$ is generated by a family of smooth vector
fields, and it is invariant with respect to their flow, that is:
\beq\label{key-eq}
(\phi_t^X)_*(\Delta_x) \subset \Delta_{\phi^X_t(x)}
\eeq
for all $x \in M$, where $\phi_t^X$ is the flow of the vector field $X$ is the given family.
\end{theorem}

The condition (\ref{key-eq})
on the flow in Stefan-Sussmann Theorem is stronger than involutivity,
see \citep{lavau} for more details.

We now turn to the example of $\cD$ the distribution on the data space
given via the DIM (Data Information Matrix) in (\ref{distr-def}).
In \citep{gf} integrability is proven in an dense open subset of the data space.

\begin{proposition}\label{gf-thm}
Consider 
a deep ReLU neural network classifier, $p$ given by softmax,
$\cD$ the DIM distribution. Let $\cH$ be the union of hyperplanes containing the singularities of 
$\cD$.
Then, 
{through} all $x \in \R^d$ $\setminus \mathcal{H}$, 
there exists a local submanifold $L \subset \R^d$, such that for all $z \in N$, 
$\cap L$, with $N$ a neighbourhood
of $x$ we have $T_zL=(\ker \cD_z)^\perp$.
\end{proposition}

In fact, since $\cD$ has constant rank in an open set of $\R^d$, 
we have that the DIM distribution
is integrable in such open set. This proposition is proven via
Frobenius Theorem  (involutivity has been proved in \cite{gf}), which however is not available for the distribution $\cD$ and its
smooth extension $\widehat{\cD}$ for the whole $\R^d$.
However we can show that $\widehat{\cD}$ is integrable as an application of the
Stefan-Sussmann Theorem.

\begin{proposition}\label{Dint}
Let the notation be as above. The singular distribution
$\widehat{\cD}$ is integrable
in the data space.
\end{proposition}

\begin{proof}
To prove this it suffices to verify the integrability condition of the Stefan-Sussmann Theorem.
Let $\cH$ be the finite union of hyperplanes as in the definition of $\widehat{\cD}$
(Def. \ref{Dhat-def}). For $x \not\in \cH$, the flow $\phi^X$ of any vector field $X \in
\widehat{\cD}$ never reaches the points in $\cH$ in a finite time. Thus, being constant
rank, smooth and involutive (Props. \ref{tf-res}, \ref{gf-thm}),
the distribution $\widehat{\cD}$ satisfies the
Stefan-Sussmann condition (\ref{key-eq}) on the open region $\R^d \setminus \cH$.
Now we look at $x \in \cH$.
Since the differential of the flow of a vector field is a linear map,
it always sends 0 to 0. So the integrability condition is trivially satisfied:
$$
(\phi^X_t)_*:0_x\mapsto 0_{\phi^X_t(x)}
$$
This proves that the distribution $\widehat{\mathcal{D}}$ is integrable.

\end{proof}

\section{Conclusions}

This paper discusses distributions and foliations emerging from the
Data Information Matrix (DIM) the analogue in the data space of the Fisher Information
Matrix. In particular, we show that in general the DIM distribution, that is,
the distribution spanned by the
probability gradients of a Deep Learning system
is non smooth, for parameters in general position in the
data space, i.e. not on a measure zero set, which is
defined by the non null sets of polynomial functions.
A suitable modification of the DIM distribution, leads to a singular smooth distribution,
satisfying the integrability conditions of the Stefan-Sussmann theorem, thus
leading a well defined notion of integral manifold in data space.
Experimental evidence in \citep{gf, tf2024}
links such integral manifolds with the dataset the model
was trained with, future investigation will furtherly investigate this relation.

\section*{Acknowledgments}
R. F. and G. N. wish to thank Prof. D. Miyamoto for valuable comments and suggestions.
R. F. and E. L. are partially supported by PNRR (PE12) MNESYS, INDAM-GNSAGA,
CaLISTA CA 21109, CaLIGOLA MSCA-2021-SE- 01-101086123, MSCA-DN CaLiForNIA - 101119552,
PNRR National Center for HPC, Big Data and Quantum Computing, INFN Sezione Bologna, SimQuSec.
E. T. benefited from financial support from the French government managed by the National Agency for Research under the France 2030 program, with the reference ANR-23-PEIA-0004.

\bibliography{main}

\appendix

\section{Smoothness of the DIM distribution}\label{smooth-app}

In this section we prove that the DIM distribution $\cD$ on the data space $\R^d$
is non smooth for the parameters
chosen in an open dense set 
in the parameter space $\R^n$. 
We use the same notation
as in Sec. \ref{distr-sec} for the score function expression (\ref{score-def}),
the set $\cH$
as in Prop. \ref{tf-res}, etc.
{
As expressed in Obs. \ref{obs-sm}, there is no change in the smoothness
of $\cD$ if we replace the distribution $\cD$ with $\cD_s$ generated by the gradients
of the scores. Hence, from now on, in this appendix,
we assume $\cD$ to be generated by the gradients of score
functions.}  

\medskip
To simplify our treatment, 
we look at points in the dataspace $\R^d$ which are not
in the intersections of two or more hyperplanes in $\cH$, the argument for generic points
being the same.

\medskip
Let $H$ be an hyperplane in $\cH$ defined by $\ReLU_k$ in the expression of
the score function at the $k$-th layer. Define $R_\pm$ the two regions of
$\mathbb{R}^d\setminus H$ as follows. {Define
$R_+$ as 
$$
R_+=\{x \in \mathbb{R}^d\setminus H \,|\, \mathbb{J}_x(\ReLU_k)_{ii}=0\}
$$
where $\mathbb{J}_x(\ReLU_k)$ is the Jacobian of the $\ReLU_k$ function at $x$.}
Similarly define $R_-$, taking those $x$ with $\mathbb{J}_x(\ReLU_k)_{ii}=1$.
Since we assume ReLU functions have value $1$ on their non-smooth point, we have
that $\mathbb{J}_x(\ReLU_k)_{ii}=1$ on $H$. We define $\mathbb{J}_x(s)_\pm$
in $H \cup R_\pm$ as follows.
$\mathbb{J}_x(s)_\pm = \mathbb{J}_x(s)$ for $x \in R_\pm$. Then, for $x \in H$,
$\mathbb{J}_x(s)_+$ in its expression has the factor $\mathbb{J}_x(\ReLU_k)$ with
$\mathbb{J}_x(\ReLU_k)_{ii}=0$,
while $\mathbb{J}_x(s)_-$ in its expression has the factor $\mathbb{J}_x(\ReLU_k)$
with $\mathbb{J}_x(\ReLU_k)_{ii}=1$.
From the very definitions, we have $\mathbb{J}_x(s)\big|_H=\mathbb{J}_x(s)_-$
{and in general $\mathbb{J}_x(s)_-\neq
\mathbb{J}_x(s)_+$.}

\begin{observation}\label{obsJ}
Notice that the rows of $\mathbb{J}_x(s)_\pm$ are smooth in their respective domains $H \cup R_\pm$.
Indeed they are the smooth vector fields $\nabla_x s_{\mu}(w)$, which are constant as $x$
varies in $\R^d$, for given fixed parameters $w \in \R^n$ (the score depends linearly on $x$).
We denote the rows of $\mathbb{J}_x(s)_\pm$ as $\nabla_x s_{\mu,\pm}$. We clearly have
$\nabla_x s_{\mu}=\nabla_x s_{\mu,\pm}$ for $x$ in $R_\pm$,
while $\nabla_x s_{\mu}=\nabla_x s_{\mu,-}$ for $x$ in $H$.
Moreover, by the very definitions, we have that the DIM distribution
$\cD$ in $H \cup R_-$ is spanned by
$\{\nabla_x s_{\mu,-}\}$ both point-wise and locally in $R_-$ only.
Similarly, in $R_+$ $\{\nabla_x s_{\mu,+}\}$ span $\cD$ both point-wise and locally.
\end{observation}

\medskip
\begin{lemma}\label{key-lem}
Let $H$ be an hyperplane and $R_\pm$ the two regions of
$\mathbb{R}^d\setminus H$. 
Let $\mathbb{J}_x(s)_\pm$ be 
as above.
Then, if the distribution $\cD$ is smooth we have:
$$
\im(\mathbb{J}_x(s)^T_+)\supseteq \im(\mathbb{J}_x(s)^T_-)\hbox{ for all points } x \in  H
$$
\end{lemma}

\begin{proof}
Suppose $\cD$ is smooth. Let $x$ be a point on the
hyperplane $H$ and $U$ be an open neighborhood of $x$ in $\R^d$.
By contradiction assume: 
$$
\hbox{Im}(\mathbb{J}_x(s)_+^T)\nsupseteq\hbox{Im}(\mathbb{J}_x(s)_-^T)
$$
Then there is some vector $v_x$, such that $v_x\in \hbox{Im}(\mathbb{J}_x(s)_-^T)$ and
$v_x\notin\hbox{Im}(\mathbb{J}_x(s)_+^T)$. 
Note that $\hbox{Im}(\mathbb{J}_x(s)_-^T)=\cD_x$.
We now show this leads to a contradiction.

Since the distribution is smooth, there exists a smooth vector field
$Y:U\rightarrow{ T\R^d}$, in the distribution $\cD|_U$, 
such that $Y(x)=v_x$. {Since $Y \in \cD|_U$, we have that $Y(y) \in
\cD_y=\hbox{Im}(\mathbb{J}_y(s)^T_+)$ for all $y \in U\cap R_+$ and $Y(z) \in \cD_z=\hbox{Im}(\mathbb{J}_z(s)^T_-)$ for all $z \in U\cap R_-$.} Notice
also that $\cD_y=\hbox{Im}(\mathbb{J}_y(s)^T_+)$ is constant as $y$ varies in
$U \cap R_+$ {(and similarly for $\cD_z$, $z\in U \cap R_-$)}. By a small
abuse of notation, identifying each tangent space at any $y$ with $\R^d$,
via canonical global coordinates, we can write
\beq\label{key-eq2}
\hbox{Im}(\mathbb{J}(s)^T_+) \subset \R^d \cong T_y\mathbb{R}^d
\eeq
dropping the index $y$, to mark the fact
it does not depend on $y$.
Hence, (\ref{key-eq2}) makes sense also for $y \in U \cap(R_+ \cup H)$,
not just for  $y \in U \cap R_+$.
We now look at:
$$
Y^{-1}(\hbox{Im}(\mathbb{J}(s)_+^T):=\{ u \in U\cap(R_+ \cup H)\, |\, Y(u) \in \hbox{Im}(\mathbb{J}(s)^T_+)\}
$$
Since $\hbox{Im}(\mathbb{J}(s)_+^T$ is a closed subspace of $\R^d$ also
$Y^{-1}(\hbox{Im}(\mathbb{J}(s)_+^T)$
is closed ($Y$ is continuous) and it contains $U \cap R_+$. Since the latter is open, it contains
also its closure, so that $U \cap H \subset
Y^{-1}(\hbox{Im}(\mathbb{J}(s)_+^T)$.
This is a contradiction since $x \in U \cap H$, so
$Y(x) \in \hbox{Im}(\mathbb{J}(s)_+^T)$,
but we are assuming $Y(x)=v_x \not\in \hbox{Im}(\mathbb{J}(s)_+^T)$.
\end{proof}

\begin{corollary}\label{cor_1}
Let $H$ be an hyperplane in $\cH$. Then $\mathbb{J}_x(s)_+$ differs from
$\mathbb{J}_x(s)_-$ by setting a set of rows of weights to 0 in $L_{W_{k+1}}$.
\end{corollary}
\begin{proof}
Recall the expression of the score function as in (\ref{score-def}):
\beq
s(x,w) = L_{W_\ell} \circ \ReLU_{\ell-1} \circ \dots \circ \ReLU_1 \circ L_{W_1}
\eeq
Setting the $n^{th}$ diagonal element of $\mathbb{J}(\mathrm{ReLU}_k)$ from 1 to
0 has the same effect of setting the $n^{th}$ row of $L_{W_{k+1}}$ to 0.
\end{proof}

\begin{observation}\label{obs_7}
Suppose that the DIM distribution $\cD$ is smooth.
Then, by Lemma \ref{key-lem} we have
$$
\im(\mathbb{J}_x(s)_+^T)\supseteq\im(\mathbb{J}_x(s)_-^T)
$$
From now on,
we assume there exists at least one $x \in \R^d$ for which the inequality
is strict. This is the same as assuming the distribution is non constant {i.e. not spanned by the columns of a constant matrix},
which is somehow a degenerate situation, which is not observed experimentally. 
{It is important not to confuse \textit{non constant} with non regular. A regular
distribution is a constant rank distribution, while a constant distribution
is one spanned by the same subspace at all points in euclidean space.}
\end{observation}

\begin{theorem}\label{not_smoot}
Let $\cD$ be the DIM distribution and assume $\cD$ is non constant.
The set of points $w\in \mathbb{R}^W$ in the parameter space for
which the distribution $\cD$ is smooth
is contained in a closed set of measure zero.
\end{theorem}
\begin{proof}
We will use the rank constraints on the DIM distribution $\cD$ to construct
polynomial equations in the parameter space, defining a closed set of measure
zero where $\cD$ can be smooth. Let $H$ be an hyperplane in $\cH$ and fix
$x$ in $H$. By Lemma
\ref{key-lem} we have that:
\beq\label{rk-eq}
\im(\mathbb{J}_x(s)_{+}^T)\supseteq \im(\mathbb{J}_x(s)_{-}^T)
\eeq
Let us choose $H$ and $x$
so that the inclusion is strict (see Obs. \ref{obs_7}).
Assume $\rank(\mathbb{J}_x(s)_{-})$ to be $k-1$; then all $k\times k$ minors
of $\mathbb{J}_x(s)_{-}$ are zero. If they do not vanish identically they
define the required closed and measure zero set in the parameter space.
If they do vanish identically,
from Corollary \ref{cor_1}, it follows that all the $k\times k$ minors
of $\mathbb{J}_x(s)_{+}$ vanish identically as well.
But this leads to a contradiction since we are assuming:
\beq
\im(\mathbb{J}_x(s)_{+}^T)\supset \im(\mathbb{J}_x(s)_{-}^T)
\eeq
hence $\mathbb{J}_x(s)_{+}$ has strictly higher rank.
\end{proof}

\section{Distributions}\label{app-distr}

In this appendix, we compare our definition of distribution (see Def. \ref{distrib-def}) on a manifold $M$ with another
notion of distribution, which is encountered in the literature and
is not equivalent to the one in Def.\ref{distrib-def}. In order to distinguish between the two definitions,
we refer to this new notion as {\sl module distribution}. This is the
approach taken in \citep{hermann, lavau}.

\begin{definition}\label{mod-distr}
{\rm We define \textit{module distribution} 
$\cF$ as a $C^\infty(M)$ submodule of $\fX(M)$ the
$C^\infty(M)$ module of smooth vector fields of $M$.}
\end{definition}

Notice that we make no hypothesis on $\cF$ being a finitely
generated or a free $C^\infty(M)$ submodule of $\fX(M)$.

The next observation clarifies the relation between the
two  different approaches to the theory of distributions. 

\begin{observation}
If $\cD$ is a smooth distribution, then we can determine a module distribution
by looking at the smooth global sections of $\cD$,
$$
\Gamma(\cD):=\{f:M \mapsto \coprod_{x\in M} \cD_x\, |\, f(x)\in \cD_x\}
$$
Similarly, if $\cF$ is a module distribution (they are by definition smooth),
we can obtain a distribution by the assignment:
$$
x \mapsto \Span\{X_x \, | \, X \in \cF\}
$$
{These two maps are not inverse of each other.}
There are many module distributions which correspond to the same
distribution and some module distributions that are not $\Gamma(\cD)$
for any distribution $\cD$.
\end{observation}

Motivated by the previous observation, we give the
following definition.

\begin{definition}
{\rm Let $\cF$ be a module distribution. We define $T\cF$ the distribution
obtained as:
$$
x \mapsto T\cF_x:=\Span\{X_x \, | \, X \in \cF\}
$$}
\end{definition}

\begin{example}\label{not-fg}
{\rm Consider the module distributions on $M=\R$ as $k \geq 0$ varies: 
$$
\cF_k=C^\infty(M)\langle x^k\partial_x \rangle \subset \fX(M) 
$$
{where $C^\infty(M)\langle x^k\partial_x \rangle=\{\sum_k f_k(x) x^k\partial_x \,|\, f_k(x) \in C^\infty(M)\}$.}

Clearly $\cF_k \not\cong \cF_l$ for $k \neq l$. However one
sees right away that
$$
\cD: x \mapsto (\cF_k)_x:=\Span\{X_x \,| \, X \in \cF_k\}
$$
give the same distribution for all $k$.

Vice-versa, by looking at the global sections of $\cD$ we retrieve
only $\cF_0$, and one can easily be convinced that the
module distributions $\cF_k$ are not retrieved as the global sections
of any distribution (this is because the information of the $x^k$ is lost
in the evaluation at $x$).}
\end{example}

We now face the question of integrability in the
two different approaches. Before we start the discussion it is useful
to keep in mind that the theorem we are about to state, Hermann's theorem, is the strongest integrability result, 
but regards only submodules of $\fX(M)$, i.e. module distributions. For this reason it is
somehow less applicable in concrete questions.

\begin{theorem}[\citet{hermann}]
Let $\Delta$ be a distribution such that $\Delta=T\cF$ for
$\cF$ a module distribution. If
\begin{enumerate}
\item $\cF$ is finitely generated as $C^\infty(M)$ submodule of $\fX(M)$
i.e.
$$
\cF=C^\infty(M)\langle X^1,\dots, X^r\rangle \subset \fX(M)
$$
\item $\cF$ is a Lie subalgebra of  $\fX(M)$, i.e.
$[X,Y] \in \cF$
\end{enumerate}
then $\Delta$ is integrable.
\end{theorem}

This result is apparently simple, but it is very
profound and hides subtle points. 

One natural question is whether a given $\Delta$ is finitely generated. In other words
if we can find global vector fields $X^1, \dots, X^n$ such that
$\Delta_x=\Span\{X^1_x, \dots, X^n_x\}$.
The answer is yes, however the
proof is very technical and surprisingly complicated \citep{dra}.

Another question regards module distributions.
Since any distribution $\Delta$ is finitely generated (see point (1)),
we can ask if $\Gamma(\Delta)
\subset \fX(M)$ is finitely generated. It is true that in general
a module distribution is not a finitely generated module, but perhaps
those coming as the global sections of a distribution
are finitely generated? 
The answer is surprisingly no. Let us look again at the example in $M=\R$.
We define the distribution $\Delta$ as follows:
$$
\Delta_x:=\Span\{X_x\}, \qquad
X=\begin{cases} 0 & \text{for $x\leq 0$} \\ \partial_x & \text{for $x > 0$}
\end{cases}
$$
$\Delta$ is smooth, since $\Delta_x$ is equivalently spanned at each
$x$ by
$$
X'=\begin{cases}  0 & \text{for $x\leq 0$} \\
e^{-\frac{1}{x^2}}\partial_x & \text{for $x > 0$}
\end{cases}
$$
One sees that:
$$
\Gamma(\Delta)\cong \{f:M=\R \lra \R\,|\, f|_{(-\infty,0]}=0\}
$$
and this is not finitely generated. It contains the functions 
{$$f(x)=\begin{cases}
		x^k\hbox{ for }x>0;\\
		0\hbox{ otherwise}
	\end{cases}$$}
for all $k${$>0$}, which are not obtainable as combination of other $C^\infty$ functions.
This fact is not entirely obvious.

Hence Hermann condition (1) does not apply, though condition (2) is
obvious (we are in dimension 1, everything is involutive). {However, }
one can see that this distribution is 
integrable.

\end{document}